\documentclass[12pt, reqno]{amsart}
\usepackage{graphicx}
\usepackage{hyperref}
\usepackage[caption = false]{subfig}
\usepackage{amsthm}
\usepackage{amssymb}
\usepackage{mathrsfs}
\usepackage{mathtools}
\usepackage{enumerate}
\usepackage{amssymb}
\usepackage{wrapfig}
\usepackage{float}
\allowdisplaybreaks

\usepackage[twoside,paperwidth=190mm, paperheight=297mm, top=35mm, bottom=20mm, left=20mm, right=20mm, marginparsep=3mm, marginparwidth=40mm]{geometry}

\numberwithin{equation}{section}

\theoremstyle{plain}

\newtheorem{theorem}{Theorem}[section]
\newtheorem{corollary}[theorem]{Corollary}

\newtheorem{lemma}{Lemma}[section]

\theoremstyle{definition}

\theoremstyle{remark}
\newtheorem{remark}{Remark}[section]

\makeatother

\begin{document}

\title{ $\mathcal{S}^{*}(\phi)$ and $\mathcal{C}(\phi)$-radii for some special functions}
	\thanks{The work of the second author is supported by University Grant Commission, New-Delhi, India  under UGC-Ref. No.:1051/(CSIR-UGC NET JUNE 2017).}	
	
	\author[S. Sivaprasad Kumar]{S. Sivaprasad Kumar}
	\address{Department of Applied Mathematics, Delhi Technological University,
		Delhi--110042, India}
	\email{spkumar@dce.ac.in}

	\author[Kamaljeet]{Kamaljeet Gangania}
	\address{Department of Applied Mathematics, Delhi Technological University,
		Delhi--110042, India}
	\email{gangania.m1991@gmail.com}

\maketitle	
	
\begin{abstract} 
In this paper, we consider the Ma-Minda classes of analytic functions 
$\mathcal{S}^{*}(\phi):= \{f\in \mathcal{A} : ({zf'(z)}/{f(z)}) \prec \phi(z) \}$ and 
$\mathcal{C}(\phi):= \{f\in \mathcal{A} : (1+{zf''(z)}/{f'(z)}) \prec \phi(z) \}$ defined on the unit disk $\mathbb{D}$ and show that the classes $\mathcal{S}^{*}(1+\alpha z)$ and $\mathcal{C}(1+\alpha z)$, $0<\alpha \leq 1$ solve the problem of finding the sharp $\mathcal{S}^{*}(\phi)$-radii and $\mathcal{C}(\phi)$-radii  for some normalized special functions, whenever $\phi(-1)=1-\alpha$. Radius of strongly starlikeness is also considered.
\end{abstract}
\vspace{0.5cm}
	\noindent \textit{2010 AMS Subject Classification}. Primary 30C45, Secondary 30C80.\\
	\noindent \textit{Keywords and Phrases}. Subordination, Ma-Minda Class, Starlikeness, Convexity, Special functions.

\maketitle
	
		\section{Introduction}
	Let  $\mathbb{D}_{r}:=\{z: |z|<r \}$ and $\mathcal{A}$ denote the class of analytic functions of the form $f(z)=z+\sum_{k=2}^{\infty}a_kz^k$ in the open unit disk $\mathbb{D}:=\mathbb{D}_{1}=\{z: |z|<1\}$. Ma-Minda \cite{minda94} introduced the following subclasses of starlike and convex functions respectively given by:
	\begin{equation}\label{mindaclass}
	\mathcal{S}^*(\phi):= \biggl\{f\in \mathcal{A} : \dfrac{zf'(z)}{f(z)} \prec \phi(z) \biggl\} \; \text{and} \;  \mathcal{C}(\phi):= \biggl\{f\in \mathcal{A} : 1+\dfrac{zf''(z)}{f'(z)} \prec \phi(z) \biggl\},
	\end{equation}
	where  $\phi$, a Ma-Minda function, is analytic and univalent with $\Re{\phi(z)}>0$, $\phi'(0)>0$ and $\phi(\mathbb{D})$ is starlike with respect to $\phi(0)=1$ and symmetric about real axis. Note that $\phi \in \mathcal{P}$, the class of normalized Carath\'{e}odory functions. If we choose $\phi(z)=(1+Dz)/ (1+Ez)$, where $-1\leq E<D\leq1$, then $\mathcal{S}^*(\phi)$ and $\mathcal{C}(\phi)$ reduces to the Janowski classes~\cite{janow} of starlike and convex functions denoted by $\mathcal{S}^*[D,E]$ and $\mathcal{C}[D,E]$ respectively. Here for the specific choices of $D$ and $E$ these classes reduces to the following classes:
	\begin{enumerate}[$(i)$]
		\item $ \mathcal{S}^*(\gamma):=\mathcal{S}^*[1-2\gamma,-1]$ and $\mathcal{C}(\gamma):=\mathcal{C}[1-2\gamma,-1]$, where $0\leq\gamma<1$ which represent the classes of starlike and convex functions of order $\gamma$, that is,
		\begin{equation*}
		\mathcal{S}^*(\gamma)= \biggl\{f\in \mathcal{A} : \Re \dfrac{zf'(z)}{f(z)} >\gamma \biggl\} \; \text{and} \;  \mathcal{C}(\gamma)= \biggl\{f\in \mathcal{A} : \Re\left( 1+\dfrac{zf''(z)}{f'(z)} \right) >\gamma \biggl\}.
		\end{equation*}
		
		\item $\mathcal{S}^*[\alpha,0]$ and $\mathcal{C}[\alpha,0]$, which are the extensions of the Ram singh~\cite{singh-1968} classes $\mathcal{S}^*[1,0]$ and $\mathcal{C}[1,0]$ respectively, where $0<\alpha\leq1$.
	\end{enumerate}  
	
	A real entire function $L$ maps real line into itself is said to be in the Laguerre-P\'{o}lya class $\mathcal{LP}$, if it can be expressed as follows:
	\begin{equation*}
	L(x)=cx^m e^{-ax^2+\beta x} \prod_{k\geq1}\left(1+\dfrac{x}{x_k} \right) e^{-\dfrac{x}{x_k}},
	\end{equation*}
	where $c,\beta,x_k\in \mathbb{R}$, $a\geq0$, $m\in \mathbb{N}\cup \{0 \}$ and $\sum {x_k}^{-2}< \infty$, see \cite{bdoy-2016}, \cite[p.~703]{lp} and the references therein. The class $\mathcal{LP}$ consists of entire functions which can be approximated by polynomials with only real zeros, uniformly on the compacts sets of the complex plane and it is closed under differentiation. Recall that $\mathcal{S}^*(\phi)$-radius for a given normalized function $f$ in $\mathcal{A}$ is defined as the largest radius $r_0$ such that $f\in \mathcal{S}^*(\phi) $ in $|z|\leq r_0$. Similarly $\mathcal{C}(\phi)$-radius can be defined. For more radius problems we refer to \cite{sinefun,ganga1997,bohr,mendi2exp}. Recently, $\mathcal{S}^*(\gamma)$-radius and $\mathcal{C}(\gamma)$-radius for the normalized special functions, which can be represented as Hadamard factorization \cite{Levin-1996} under certain conditions were studied~\cite{bdoy-2016,abo-2018}.  Some of the special functions, which are studied recently are Bessel functions~\cite{abo-2018} (see Watson's treatise~\cite{watson-1944} for more on Bessel function), Struve functions~\cite{bdoy-2016,abo-2018}, Wright functions~\cite{btk-2018}, Lommel functions~\cite{bdoy-2016,abo-2018}, Legendre polynomials of odd degree~\cite{bulut-engel-2019} and Ramanujan type entire functions~\cite{ErhanDenij2020}. Evidently, the zeros of these special functions and the $\mathcal{LP}$ class played an important role in the derivation of the above radius results. 
	
	In this paper, we study the $\mathcal{S}^*(\phi)$-radius and $\mathcal{C}(\phi)$-radius problems of certain special functions using the following extension of the Ram singh class:
	\begin{equation*}
	\mathcal{S}^*(1+\alpha z)= \biggl\{f\in \mathcal{A} : \dfrac{zf'(z)}{f(z)} \prec 1+\alpha z \biggl\}
	\end{equation*}	
	and
	\begin{equation*}
	  \mathcal{C}(1+\alpha z)= \biggl\{f\in \mathcal{A} : 1+\dfrac{zf''(z)}{f'(z)} \prec 1+\alpha z \biggl\},
	\end{equation*}
	where $0<\alpha\leq1$. We consider the case $\phi(-1)=1-\alpha$, which covers many classical classes and the recently introduced classes, see Corollary~\ref{application}. Moreover, our results also hold for the Ma-Minda Janowski and the Lemniscate of Bernoulli classes. The radius of strongly starlikeness is also considered. We assume $\{w: |w-1|<\alpha\}$ is the maximal disk inside $\phi(\mathbb{D})$ and $\phi(-1)=1-\alpha$, in what follows.
	
	\section{Wright and Mittag-Leffler functions}
	We deal here with two special functions.
	\subsection{Wright functions}\label{sec-1}
		Let us consider the generalized Bessel function
	\begin{equation*}
	\Phi(\rho, \beta, z) = \sum_{n\geq0} \dfrac{z^n}{n! \Gamma(n\rho+\beta)},
	\end{equation*}
	where $\rho>-1$ and $z, \beta \in \mathbb{C}$, named after E. M. Wright. The function $\Phi$ is entire for $\rho>-1$. From \cite[Lemma~1, p.~100]{btk-2018}, we have the following Hadamard  factorization
	\begin{equation}\label{Had-wrt}
	\Gamma(\beta) \Phi(\rho, \beta,-z^2)= \prod_{n\geq1}\left(1-\dfrac{z^2}{{\zeta}^{2}_{\rho,\beta,n}}\right),
	\end{equation}
	where $\rho, \beta>0$ and ${\zeta}_{\rho,\beta,n}$ is the $n$-th positive root of $\Phi(\rho,\beta,-z^2)$ and satisfies the following relationship:
	\begin{equation}\label{w-roots}
	\breve{\zeta}_{\rho, \beta, n}< {\zeta}_{\rho, \beta, n} < \breve{\zeta}_{\rho, \beta, n+1}< {\zeta}_{\rho, \beta, n+1}, \quad (n\geq1)
	\end{equation}
	where $\breve{\zeta}_{\rho,\beta,n}$ is the $n$-th positive root of the derivative of the function $\Psi_{\rho, \beta}(z)=z^{\beta}\Phi(\rho, \beta,-z^2)$. Since $\Phi(\rho, \beta, -z^2)\not \in \mathcal{A}$, so we consider the following normalized Wright functions:
	\begin{align}\label{w1}
	f_{\rho, \beta}(z) &= \left[z^{\beta}\Gamma(\beta) \Phi(\rho, \beta, -z^2)\right]^{1/\beta} \nonumber\\
	g_{\rho, \beta}(z) &= z\Gamma(\beta) \Phi(\rho, \beta, -z^2) \nonumber\\
	h_{\rho, \beta}(z) &= z\Gamma(\beta) \Phi(\rho, \beta, -z).
	\end{align}
	For simplicity, we write $W_{\rho, \beta}(z):=\Phi(\rho, \beta, -z^2)$. 
	
	\begin{theorem}\label{wright-star}
		Let $\rho, \beta>0$. Then $\mathcal{S}^{*}(1+\alpha z)$-radii for the functions $f_{\rho,\beta}$, $g_{\rho,\beta}$ and $h_{\rho,\beta}$ are the smallest positive roots of the following equations respectively:
		\begin{enumerate}[$(i)$]
			\item $r W^{'}_{\rho,\beta}(r)+\beta \alpha W_{\rho, \beta}(r)=0$;
			
			\item $r W^{'}_{\rho,\beta}(r)+ \alpha W_{\rho, \beta}(r)=0$;
			
			\item $\sqrt{r} W^{'}_{\rho,\beta}(\sqrt{r})+ 2\alpha W_{\rho, \beta}(\sqrt{r})=0$,
			
		\end{enumerate}
		where $\alpha$ is the radius of disk $\{w: |w-1|\leq \alpha\}$.
	\end{theorem}
	\begin{proof}
		Using \eqref{Had-wrt}, we obtain the following by the logarithmic differentiation of \eqref{w1}:
		\begin{align}\label{w-sharp}
		\dfrac{zf'_{\rho,\beta}(z)}{f_{\rho,\beta}(z)} &= 1+ \dfrac{1}{\beta} \dfrac{z W'_{\rho, \beta}(z)}{W_{\rho, \beta}(z)}=1-\dfrac{1}{\beta} \sum_{n\geq1}\dfrac{2z^2}{{\zeta}^2_{\rho,\beta,n} -z^2} \nonumber\\
		\dfrac{zg'_{\rho,\beta}(z)}{g_{\rho,\beta}(z)} &= 1+ \dfrac{z W'_{\rho, \beta}(z)}{W_{\rho, \beta}(z)}=1- \sum_{n\geq1}\dfrac{2z^2}{{\zeta}^2_{\rho,\beta,n} -z^2} \nonumber\\
		\dfrac{zh'_{\rho,\beta}(z)}{h_{\rho,\beta}(z)} &= 1+ \dfrac{1}{2} \dfrac{\sqrt{z} W'_{\rho, \beta}(\sqrt{z})}{W_{\rho, \beta}(\sqrt{z})}=1- \sum_{n\geq1}\dfrac{z}{{\zeta}^2_{\rho,\beta,n} -z}. 
		\end{align}
		Now using the fact that $||x|-|y||\leq |x-y|$ and $|z|=r< {\zeta}_{\rho,\beta,1}$, we see that $f_{\rho,\beta}$, $g_{\rho,\beta}$ and $h_{\rho,\beta}$ belong to $\mathcal{S}^{*}(1+\alpha z)$ whenever
		\begin{align}\label{w-mod}
		\left|\dfrac{zf'_{\rho,\beta}(z)}{f_{\rho,\beta}(z)} -1\right| &\leq\dfrac{1}{\beta} \sum_{n\geq1}\dfrac{2r^2}{{\zeta}^2_{\rho,\beta,n} -r^2} \leq \alpha \nonumber\\
		\left| \dfrac{zg'_{\rho,\beta}(z)}{g_{\rho,\beta}(z)} -1\right| &\leq \sum_{n\geq1}\dfrac{2r^2}{{\zeta}^2_{\rho,\beta,n} -r^2} \leq \alpha \nonumber\\
		\left|\dfrac{zh'_{\rho,\beta}(z)}{h_{\rho,\beta}(z)} -1\right| &\leq \sum_{n\geq1}\dfrac{r}{{\zeta}^2_{\rho,\beta,n} -r} \leq \alpha. 
		\end{align}
		The first part of each of the inequalities in \eqref{w-mod} becomes equality when $z=r$. Now consider the following continuous functions:
		\begin{align*}
		T_{f}(r) = \dfrac{1}{\beta} \sum_{n\geq1}\dfrac{2r^2}{{\zeta}^2_{\rho,\beta,n} -r^2} -\alpha, \quad 
		T_{g}(r) =  \sum_{n\geq1}\dfrac{2r^2}{{\zeta}^2_{\rho,\beta,n} -r^2} -\alpha, \quad 
		T_{h}(r) = \sum_{n\geq1}\dfrac{r}{{\zeta}^2_{\rho,\beta,n} -r} -\alpha.
		\end{align*} 
		Note that $T_{f}$, $T_{g}$ are increasing in $(0,  {\zeta}_{\rho,\beta,1})$ and $T_{h}$ is in $(0,  {\zeta}^2_{\rho,\beta,1})$. Since $\lim_{r\rightarrow 0}T_{f}(r)=\lim_{r\rightarrow 0}T_{g}(r)=\lim_{r\rightarrow 0}T_{h}(r)=-\alpha<0$ and $\lim_{r\rightarrow {\zeta}_{\rho,\beta,1}}T_{f}(r)=\lim_{r\rightarrow {\zeta}_{\rho,\beta,1}}T_{g}(r)=\infty$, $\lim_{r\rightarrow {\zeta}^2_{\rho,\beta,1}}T_{h}(r)=\infty$. Therefore, the required $\mathcal{S}^{*}(1+\alpha z)$-radii for the functions $f_{\rho,\beta}$ and $g_{\rho,\beta}$ are the unique positive roots of the equations $T_{f}(r)=0$, $T_{g}(r)=0$ and for $h_{\rho,\beta}$ given by $T_{h}(r)=0$, respectively in $(0,  {\zeta}_{\rho,\beta,1})$ and $(0,  {\zeta}^2_{\rho,\beta,1})$, which can be written as in the statement using \eqref{w-sharp}. Let $r_{\alpha,f}$, $r_{\alpha,g}$ and $r_{\alpha, h}$ be the roots of $T_{f}(r)=0$, $T_{g}(r)=0$ and $T_{h}(r)=0$ respectively. Then from \eqref{w-sharp}, we see that
		\begin{equation*}
		\dfrac{r_{\alpha,f}f'_{\rho,\beta}(r_{\alpha,f})}{f_{\rho,\beta}(r_{\alpha,f})}=\dfrac{r_{\alpha,g}g'_{\rho,\beta}(r_{\alpha,g})}{g_{\rho,\beta}(r_{\alpha,g})}=\dfrac{r_{\alpha, h}h'_{\rho,\beta}(r_{\alpha, h})}{h_{\rho,\beta}(r_{\alpha, h})} =1-\alpha.  
		\end{equation*}
		Hence the radii are sharp.
		\qed    
	\end{proof}	
	
	\begin{remark}
		From \cite[Theorem~1]{btk-2018}, we see that the equations of Theorem~\ref{wright-star} yields the radius of starlikeness of order $\gamma:=1-\alpha$ for $f_{\rho,\beta}$, $g_{\rho,\beta}$ and $h_{\rho,\beta}$.  
	\end{remark}
	
	We denote $\mathcal{S}^{*}(\phi)$-radius by $R[\mathcal{S}^{*}(\phi)]$.
	\begin{theorem}\label{wright-phi}
		Let $\rho, \beta>0$. Then there exists an $\alpha\in (0,1]$ such that the largest disk $\{w: |w-1|< \alpha\} \subseteq \phi(\mathbb{D})$ and
		$R[\mathcal{S}^{*}(\phi)] = R[\mathcal{S}^{*}(1+\alpha z)]$
		for the functions $f_{\rho,\beta}$, $g_{\rho,\beta}$ and $h_{\rho,\beta}$ whenever $\phi(-1)=1-\alpha$.
	\end{theorem}
	\begin{proof}
		Choose $\alpha\in (0,1]$ such that $w_{\alpha}:=\{w: |w-1|<\alpha\}$ is the maximal disk inside $\phi(\mathbb{D})$. Let $r_{\alpha,f}$, $r_{\alpha,g}$ and $r_{\alpha, h}$ denote the smallest positive root of the equations given in Theorem~\ref{wright-star}. Then $f_{\rho,\beta}$, $g_{\rho,\beta}$ and $h_{\rho,\beta}$ belong to $\mathcal{S}^{*}(1+\alpha z)$ in $|z|<r_{\alpha,f}$, $r_{\alpha,g}$ and  $r_{\alpha,h}$, respectively. Since a function $f_1(z) \in \mathcal{S}^{*}(\phi)$ if and only if $e^{-it}f_1(e^{it}z) \in \mathcal{S}^{*}(\phi)$ for all $t\in \mathbb{R}$. Therefore, using \eqref{w-sharp} with $z=r_{\alpha,f}$, $r_{\alpha,g}$ and $r_{\alpha, h}$ along with $\phi(-1)=1-\alpha$, the maximality of the disk $w_{\alpha}$ implies that $F_{\rho,\beta}(|z|\leq r)$, $G_{\rho,\beta}(|z|\leq r)$ and $H_{\rho,\beta}(|z|\leq r)$ do not lie inside $\phi(\mathbb{D})$ for $r\geq r_{\alpha,f}$, $r_{\alpha,g}$ and  $r_{\alpha,h}$ respectively, where $F_{\rho,\beta}(z)=zf'_{\rho,\beta}(z)/f_{\rho,\beta}(z)$, $G_{\rho,\beta}(z)=zg'_{\rho,\beta}(z)/g_{\rho,\beta}(z)$ and $H_{\rho,\beta}(z)=zh'_{\rho,\beta}(z)/h_{\rho,\beta}(z)$ (with some suitable rotation). Hence, $f_{\rho,\beta}$, $g_{\rho,\beta}$ and $h_{\rho,\beta}$ belong to $\mathcal{S}^{*}(\phi)$ in $|z|<r_{\alpha,f}$, $r_{\alpha,g}$ and  $r_{\alpha,h}$, respectively and the radii are sharp. \qed       
	\end{proof}
	
	\begin{theorem}\label{wright-conx}
		Let $\rho, \beta>0$. Then $\mathcal{C}(1+\alpha z)$-radii for the functions $f_{\rho,\beta}$, $g_{\rho,\beta}$ and $h_{\rho,\beta}$ are the smallest positive roots of the following equations respectively:
		\begin{enumerate}[$(i)$]
			\item $ \dfrac{r {\Psi}^{''}_{\rho,\beta}(r)}{{\Psi}^{'}_{\rho, \beta}(r)}+\left(\dfrac{1}{\beta}-1\right) \dfrac{r {\Psi}^{'}_{\rho, \beta}(r)}{{\Psi}_{\rho, \beta}(r)}+\alpha=0$;
			
			\item ${rg''_{\rho,\beta}(r)}+\alpha {g'_{\rho,\beta}(r)}=0$;
			
			\item ${rh''_{\rho,\beta}(z)}+\alpha {h'_{\rho,\beta}(r)}=0$,
			
		\end{enumerate}
		where $\alpha$ is the radius of the disk $\{w: |w-1|\leq \alpha\}$.
	\end{theorem}
	\begin{proof}
		From \eqref{Had-wrt}, \eqref{w1} and the representation $\Gamma(\beta) {\Psi}^{'}_{\rho, \beta}(z)=\beta z^{\beta-1} \prod_{n\geq1}\left(1-\dfrac{z^2}{ \breve{\zeta}^2_{\rho, \beta, n}} \right)$, (see \cite[Eq.~7]{btk-2018}), we have
		\begin{align*}
		1+\dfrac{z f''_{\rho,\beta}(z)}{f'_{\rho,\beta}(z)}&=1+\dfrac{z {\Psi}^{''}_{\rho,\beta}(z)}{{\Psi}^{'}_{\rho, \beta}(z)}+ \left(\dfrac{1}{\beta}-1\right)\dfrac{z {\Psi}^{'}_{\rho,\beta}(z)}{{\Psi}_{\rho, \beta}(z)}\\
		& = 1-\sum_{n\geq1}\dfrac{2z^2}{\breve{\zeta}^2_{\rho, \beta, n} -z^2} -\left(\dfrac{1}{\beta}-1\right)\sum_{n\geq1}\dfrac{2z^2}{	{\zeta}^2_{\rho, \beta, n}-z^2}
		\end{align*}
		and for $\beta>1$, using the following inequality of \cite{Deniz-2017}: 
		\begin{equation}\label{firstnorm}
		\left|\dfrac{z}{y-z}-\lambda \dfrac{z}{x-z}\right| \leq \dfrac{|z|}{y-|z|}-\lambda \dfrac{|z|}{x-|z|},\quad (x>y>r\geq|z|)
		\end{equation}
		with $\lambda=1-1/\beta$, we get
		\begin{align*}
		\left|\dfrac{z f''_{\rho,\beta}(z)}{f'_{\rho,\beta}(z)} \right|&=\left|\sum_{n\geq1}\dfrac{2z^2}{\breve{\zeta}^2_{\rho, \beta, n} -z^2} -\left(1-\dfrac{1}{\beta}\right)\sum_{n\geq1}\dfrac{2z^2}{	{\zeta}^2_{\rho, \beta, n}-z^2}\right| \\
		&\leq \sum_{n\geq1}\dfrac{2r^2}{\breve{\zeta}^2_{\rho, \beta, n} -r^2} -\left(1-\dfrac{1}{\beta}\right)\sum_{n\geq1}\dfrac{2r^2}{	{\zeta}^2_{\rho, \beta, n}-r^2}\\
		& =
		-\dfrac{r f''_{\rho,\beta}(r)}{f'_{\rho,\beta}(r)}=-\dfrac{r {\Psi}^{''}_{\rho,\beta}(r)}{{\Psi}^{'}_{\rho, \beta}(r)}- \left(\dfrac{1}{\beta}-1\right)\dfrac{r {\Psi}^{'}_{\rho,\beta}(r)}{{\Psi}_{\rho, \beta}(r)}
		\end{align*}
		Since $g_{\rho, \beta}$ and $h_{\rho, \beta}$ belong to the Laguerre-P\'{o}lya class $\mathcal{LP}$, which is closed under differentiation, their derivatives $g'_{\rho, \beta}$ and $h'_{\rho, \beta}$ also belong to $\mathcal{LP}$ and the zeros are real. Thus assuming $\tau_{\rho,\beta,n}$  and $\eta_{\rho,\beta,n}$ are the positive zeros of $g'_{\rho, \beta}$ and $h'_{\rho, \beta}$, respectively, we have the following representations:
		\begin{align*}
		g'_{\rho,\beta}(z)=\prod_{n\geq1}\left(1-\dfrac{z^2}{{\tau}^2_{\rho,\beta,n}}\right) \quad
		\text{and} \quad
		h'_{\rho,\beta}(z)=\prod_{n\geq1}\left(1-\dfrac{z}{\eta_{\rho,\beta,n}}\right),
		\end{align*}
		which yield
		\begin{align*}
		1+\dfrac{zg''_{\rho,\beta}(z)}{g'_{\rho,\beta}(z)}=1-\sum_{n\geq1}\dfrac{2z^2}{{\tau}^2_{\rho,\beta,n}-z^2} \quad
		\text{and}\quad
		1+\dfrac{zh''_{\rho,\beta}(z)}{h'_{\rho,\beta}(z)}=1-\sum_{n\geq1}\dfrac{z}{{\eta}_{\rho,\beta,n}-z}.
		\end{align*}
		Now using the inequality $||x|-|y||\leq |x-y|$ and the  relation \eqref{w-roots}, we see that $f_{\rho,\beta}$, $g_{\rho,\beta}$ and $h_{\rho,\beta}$ belongs to $\mathcal{C}(1+\alpha z)$ whenever
		\begin{align*}
		&\left|\dfrac{z f''_{\rho,\beta}(z)}{f'_{\rho,\beta}(z)}\right| \leq \sum_{n\geq1}\dfrac{2r^2}{\breve{\zeta}^2_{\rho, \beta, n} -r^2} +\left(\dfrac{1}{\beta}-1\right)\sum_{n\geq1}\dfrac{2r^2}{{\zeta}^2_{\rho, \beta, n}-r^2} \leq \alpha, \quad (\beta>0, |z|=r< \breve{\zeta}_{\rho,\beta,1})\\
		& \left|\dfrac{zg''_{\rho,\beta}(z)}{g'_{\rho,\beta}(z)}\right|\leq\sum_{n\geq1}\dfrac{2r^2}{{\tau}^2_{\rho,\beta,n}-r^2}\leq \alpha \quad (|z|=r< {\tau}_{\rho,\beta,1})
		\quad\text{and}\\
		&\left|\dfrac{zh''_{\rho,\beta}(z)}{h'_{\rho,\beta}(z)}\right|\leq\sum_{n\geq1}\dfrac{r}{{\eta}_{\rho,\beta,n}-r}\leq  \alpha \quad (|z|=r< {\eta}_{\rho,\beta,1})
		\end{align*}
		respectively. Now further proceeding as in Theorem~\ref{wright-star}, the result follows at once. \qed
	\end{proof}

	\begin{remark}
		From \cite[Theorem~5]{btk-2018}, we see that equations of Theorem~\ref{wright-conx} yields the radius of starlikeness of order $\gamma:=1-\alpha$ for $f_{\rho,\beta}$, $g_{\rho,\beta}$ and $h_{\rho,\beta}$.  
	\end{remark}
	
	We denote $\mathcal{C}(\phi)$-radius by $R[\mathcal{C}(\phi)]$.
	\begin{theorem}\label{wright-c}
		Let $\rho, \beta>0$. Then there exists an $\alpha\in (0,1]$ such that the largest disk $\{w: |w-1|< \alpha\} \subseteq \phi(\mathbb{D})$ and
		$R[\mathcal{C}(\phi)] = R[\mathcal{C}(1+\alpha z)]$
		for the functions $f_{\rho,\beta}$, $g_{\rho,\beta}$ and $h_{\rho,\beta}$, whenever $\phi(-1)=1-\alpha$.
	\end{theorem}
	The proof of Theorem~\ref{wright-c} is similar to Theorem~\ref{wright-phi} and hence it is skipped here.

	\subsection{Mittag-Leffler functions}\label{sec-2}
	In 1971, Prabhakar~\cite{prabha-1971} introduced the following function
	\begin{equation*}
	M(\mu,\nu, a, z):= \sum_{n\geq0} \dfrac{(a)_{n} z^n}{n! \Gamma(\mu n,\nu)},
	\end{equation*}
	where $(a)_{n}=\Gamma(a+n)/\Gamma(a)$ denotes the Pochhammer symbol and $\mu, \nu,a>0$. The functions $M(\mu,\nu, 1, z)$ and $M(\mu, 1, 1, z)$ were introduced and studied by Wiman and Mittag-Leffler, respectively. Now let us consider the set $ W_{b}= A(W_{c})\cup B(W_{c})$, where
	\begin{equation*}
	W_{c}:= \left\{  \left(\dfrac{1}{\mu},\nu\right): 1<\mu<2, \nu\in [\mu-1,1]\cup[\mu,2]   \right\}
	\end{equation*}
	and denote by $W_{i}$, the smallest set containing $W_{b}$ and invariant under the transformations $A$, $B$ and $C$ mapping the set
	$\{(\tfrac{1}{\mu},\nu): \mu>1, \nu>0\}$ into itself and are defined as:
	\begin{align*}
	&A: (\tfrac{1}{\mu},\nu)\rightarrow (\tfrac{1}{2\mu},\nu),\quad B: (\tfrac{1}{\mu},\nu)\rightarrow (\tfrac{1}{2\mu},\mu+\nu), \\
	&C: (\tfrac{1}{\mu},\nu)\rightarrow 
	\left\{
	\begin{array}
	{lr}
	(\tfrac{1}{\mu},\nu-1),     & \text{if}\; \nu>1; \\
	(\tfrac{1}{\mu},\nu),     & \text{if}\; 0<\nu\leq1.
	\end{array}
	\right.
	\end{align*}
	Kumar and Pathan \cite{pathan-2016} proved that if $(\tfrac{1}{\mu},\nu)\in W_{i}$ and $a>0$, then all zeros of $M(\mu,\nu, a, z)$ are real and negative. From \cite[Lemma~1, p.~121]{b-praj-2020}, we see that if  $(\tfrac{1}{\mu},\nu)\in W_{i}$ and $a>0$, then the function $M(\mu,\nu, a, -z^2)$ has infinitley many zeros, which are all real and have the following representation:
	\begin{equation*}
	\Gamma(\nu)M(\mu,\nu, a, -z^2)=\prod_{n\geq1}\left(1-\dfrac{z^2}{{\lambda}^2_{\mu,\nu,a,n}}\right),
	\end{equation*}
	where ${\lambda}_{\mu,\nu,a,n}$ is the $n$-th positive zero of $M(\mu,\nu, a, -z^2)$ and satisfy the interlacing relation
	\begin{equation*}
	{\xi}_{\mu,\nu, a,n}< {\lambda}_{\mu,\nu,a,n}< {\xi}_{\mu,\nu, a,n+1} < {\lambda}_{\mu,\nu,a,n+1} \quad (n\geq1),
	\end{equation*}
	where ${\xi}_{\mu,\nu, a,n}$ is the $n$-th positive zero of the derivative of $ z^{\nu}M(\mu,\nu, a, -z^2)$. Since $M(\mu,\nu, a, -z^2)\not\in \mathcal{A}$, therefore we consider the following normalized forms (belong to the Laguerre-P\'{o}lya class):
	\begin{align}\label{mittag1}
	f_{\mu, \nu,a}(z) &= \left[z^{\nu}\Gamma(\nu) M(\mu, \nu, a, -z^2)\right]^{1/\nu} \nonumber\\
	g_{\mu, \nu,a}(z) &= z\Gamma(\nu) M(\mu, \nu, a, -z^2) \nonumber\\
	h_{\mu, \nu,a}(z) &= z\Gamma(\nu) M(\mu, \nu,a, -z).
	\end{align}
	
	For simplicity, write $L(\mu,\nu,a,z):=M(\mu,\nu, a, -z^2)$. Now proceeding similarly as in Section~\ref{sec-1}, we obtain the following results:
	\begin{theorem}\label{mittag-star}
		Let $(\tfrac{1}{\mu},\nu)\in W_{i}$ and $a>0$. Then $\mathcal{S}^{*}(1+\alpha z)$-radii for the functions $f_{\mu,\nu, a}$, $g_{\mu,\nu, a}$ and $h_{\mu,\nu, a}$ are the smallest positive roots of the following equations respectively:
		\begin{enumerate}[$(i)$]
			\item $r L^{'}_{\mu,\nu,a}(r)+\nu \alpha L_{\mu, \nu,a}(r)=0$;
			
			\item $r L^{'}_{\mu,\nu,a}(r)+\alpha L_{\mu, \nu,a}(r)=0$;
			
			\item $\sqrt{r} L^{'}_{\mu,\nu,a}(\sqrt{r})+ 2\alpha L_{\mu,\nu,a}(\sqrt{r})=0$,
			
		\end{enumerate}
		where $\alpha$ is the radius of the disk $\{w: |w-1|\leq \alpha\}$.
	\end{theorem}
	
	\begin{theorem}\label{mittag-phi}
		Let $(\tfrac{1}{\mu},\nu)\in W_{i}$ and $a>0$. Then there exists an $\alpha\in (0,1]$ such that the largest disk $\{w: |w-1|< \alpha\} \subseteq \phi(\mathbb{D})$ and
		$R[\mathcal{S}^{*}(\phi)] = R[\mathcal{S}^{*}(1+\alpha z)]$
		for the functions $f_{\mu,\nu, a}$, $g_{\mu,\nu, a}$ and $h_{\mu,\nu, a}$.
	\end{theorem}

	\begin{theorem}\label{mittag-conx}
		Let $(\tfrac{1}{\mu},\nu)\in W_{i}$ and $a>0$. Then $\mathcal{C}(1+\alpha z)$-radii for the functions $f_{\mu,\nu, a}$, $g_{\mu,\nu, a}$ and $h_{\mu,\nu, a}$ are the smallest positive roots of the following equations respectively:
		\begin{enumerate}[$(i)$]
			\item $ {rf''_{\mu,\nu, a}(r)}+\alpha {f'_{\mu,\nu, a}(r)}=0$;
			
			\item ${rg''_{\mu,\nu, a}(r)}+\alpha {g'_{\mu,\nu, a}(r)}=0$;
			
			\item ${rh''_{\mu,\nu, a}(r)}+\alpha {h'_{\mu,\nu, a}(r)}=0$,
			
		\end{enumerate}
		where $\alpha$ is the radius of the disk $\{w: |w-1|\leq \alpha\}$.
	\end{theorem}
	
	\begin{theorem}\label{mittag-c}
		Let $(\tfrac{1}{\mu},\nu)\in W_{i}$ and $a>0$. Then there exists an $\alpha\in (0,1]$ such that the largest disk $\{w: |w-1|< \alpha\} \subseteq \phi(\mathbb{D})$ and
		$R[\mathcal{C}(\phi)] = R[\mathcal{C}(1+\alpha z)]$
		for the functions $f_{\mu,\nu, a}$, $g_{\mu,\nu, a}$ and $h_{\mu,\nu, a}$.
	\end{theorem}
	
	\begin{remark}
		From \cite[Theorem~1, Theorem~3]{b-praj-2020}, we see that equations of Theorem~\ref{mittag-star} and Theorem~\ref{mittag-conx} yields the radius of starlikeness and radius of convexity of order $\gamma:=1-\alpha$ for $f_{\mu,\nu, a}$, $g_{\mu,\nu, a}$ and $h_{\mu,\nu, a}$, respectively.  
	\end{remark}
	
	\section{Results on Convexity}
	It is evident that $R[\mathcal{S}^{*}(\phi)]=R[\mathcal{S}^{*}(1+\alpha z)]$ for the Lommel function~\cite{bdoy-2016}, Struve functions~\cite{bdoy-2016} and odd degree Legendre polynomials~\cite{bulut-engel-2019}, which is proved in \cite{SG-2020} by us.
	\subsection{ Convexity of Legendre polynomials}
	The Legendre polynomials $P_{n}$ are the solutions of the Legendre differential equation
	$$((1-z^2)P'_{n}(z))'+n(n+1)P_{n}(z)=0,$$
	where $n\in \mathbb{Z}^{+}$ and using Rodrigues formula, $P_{n}$ can be represented in the form:
	$$P_{n}(z)=\dfrac{1}{2^n n!}\dfrac{d^n(z^2-1)^n}{dz^n}$$
	and it also satisfies the geometric condition $P_n(-z)=(-1)^n P_{n}(z)$. Moreover, the odd degree Legendre polynomials $P_{2n-1}(z)$  have only real roots which satisfy 
	\begin{equation}\label{legdroot}
	0=z_0<z_1<\cdots<z_{n-1}\quad\text{or}\quad -z_1>\cdots>-z_{n-1}.
	\end{equation}
	Thus the normalized form is as follows:
	\begin{equation}\label{legd1}
	\mathcal{P}_{2n-1}(z):=\dfrac{P_{2n-1}(z)}{P'_{2n-1}(0)}=z+\sum_{k=2}^{2n-1}a_{k}z^{k}=a_{2n-1}z\prod_{k=1}^{n-1}(z^2-z^2_{k}).
	\end{equation}

	\begin{theorem}\label{Leg-c}
		$R[\mathcal{C}(\phi)]=R[\mathcal{C}(1+\alpha z)]$ for the normalized Legendre polynomial of odd degree is given by the smallest positive root $r(\mathcal{P}_{2n-1})$ of the equation
		\begin{equation*}
		{r \mathcal{P}''_{2n-1}(r)} +\alpha {\mathcal{P}'_{2n-1}(r)}=0,
		\end{equation*} 
		where $\alpha$ is the radius of the largest disk $\{w: |w-1|< \alpha\}$ inside $\phi(\mathbb{D})$.	
	\end{theorem}
	\begin{proof}
		From \eqref{legd1}, we obtain
		\begin{equation*}
		1+\dfrac{z \mathcal{P}''_{2n-1}(z)}{\mathcal{P}'_{2n-1}(z)}=\dfrac{z \mathcal{P}'_{2n-1}(z)}{\mathcal{P}_{2n-1}(z)} -\dfrac{\sum_{k=1}^{n-1}\dfrac{4z^2_k z^2}{(z^2_k -z^2)^2}}{\dfrac{z \mathcal{P}'_{2n-1}(z)}{\mathcal{P}_{2n-1}(z)}} =1-\sum_{k=1}^{n-1}\dfrac{2z^2}{z^2_k-z^2} -\dfrac{\sum_{k=1}^{n-1}\dfrac{4z^2_k z^2}{(z^2_k -z^2)^2}}{1-\sum_{k=1}^{n-1}\dfrac{2z^2}{z^2_k-z^2}},
		\end{equation*}
		which implies, after using the inequality $||x|-|y||\leq |x-y|$ and \eqref{legdroot} for $ |z|=r<z_1$
		\begin{equation}\label{mod-leg}
		\left|\left(1+\dfrac{z \mathcal{P}''_{2n-1}(z)}{\mathcal{P}'_{2n-1}(z)} \right)-1 \right| \leq \sum_{k=1}^{n-1}\dfrac{2r^2}{z^2_k-r^2} +\dfrac{\sum_{k=1}^{n-1}\dfrac{4z^2_k r^2}{(z^2_k -r^2)^2}}{1-\sum_{k=1}^{n-1}\dfrac{2r^2}{z^2_k-r^2}}= -\dfrac{r \mathcal{P}''_{2n-1}(r)}{\mathcal{P}'_{2n-1}(r)}.
		\end{equation}
		Now let $\alpha$ be the largest such that $\{w: |w-1|\leq \alpha\} \subseteq \phi(\mathbb{D})$. Then from \eqref{mod-leg}, we see that $\mathcal{P}_{2n-1}\in \mathcal{C}(1+\alpha z)\subseteq \mathcal{C}(\phi)$, whenever
		\begin{equation*}
		{r \mathcal{P}''_{2n-1}(r)} +\alpha {\mathcal{P}'_{2n-1}(r)} \geq 0,
		\end{equation*}
		which holds in $|z|=r\leq r(\mathcal{P}_{2n-1})$. Sharpness of the radius $r(\mathcal{P}_{2n-1})$ follows from the suitable rotation of $\mathcal{P}_{2n-1}$. \qed
	\end{proof}

	\subsection{ Convexity of Lommel functions}
	The Lommel function $\mathcal{L}_{u,v}$ of first kind is a particular solution of the second-order inhomogeneous Bessel differential equation $$z^2w''(z)+zw'(z)+(z^2-{v}^2)w(z)=z^{u+1},$$
	where $u\pm v\notin \mathbb{Z}^{-}$ and is given by
	$$\mathcal{L}_{u,v}=\dfrac{z^{u+1}}{(u-v+1)(u+v+1)}{}_1F_2\left(1;\dfrac{u-v+3}{2},\dfrac{u+v+3}{2};-\dfrac{z^2}{4}\right),$$
	where $\dfrac{1}{2}(-u\pm v-3)\notin \mathbb{N}$ and ${}_1 F_{2}$ is a hypergeometric function. Since it is not normalized, so we consider the following three normalized functions involving $\mathcal{L}_{u,v}$ :
	\begin{align}\label{fL}
	f_{u,v}(z)&=((u-v+1)(u+v+1)\mathcal{L}_{u,v}(z))^{\tfrac{1}{u+1}},\nonumber\\
	g_{u,v}(z)&=(u-v+1)(u+v+1)z^{-u}\mathcal{L}_{u,v}(z),\nonumber\\
	h_{u,v}(z)&=(u-v+1)(u+v+1)z^{\dfrac{1-u}{2}}\mathcal{L}_{u,v}(\sqrt{z}).
	\end{align}
	Authors in \cite{abo-2018,bdoy-2016} and \cite{Bricz-Rama} proved the radius of starlikeness and convexity for the following normalized functions expressed in terms of $\mathcal{L}_{u-\tfrac{1}{2},\tfrac{1}{2}}$:
	\begin{equation}\label{lomel-normalized}
	f_{u-\tfrac{1}{2},\tfrac{1}{2}}(z),\quad g_{u-\tfrac{1}{2},\tfrac{1}{2}}(z) \quad\text{and}\quad h_{u-\tfrac{1}{2},\tfrac{1}{2}}(z),
	\end{equation} 
	where $0\neq u\in (-1,1)$.
	Now we find $R[\mathcal{C}(\phi)]$  of the functions defined in \eqref{lomel-normalized}. For simplicity, we write these as $f_{u}, g_{u}$ and $h_{u}$, respectively and $\mathcal{L}_{u-\tfrac{1}{2},\tfrac{1}{2}}=\mathcal{L}_{u}$.
	
	\begin{theorem}\label{lommeltheorem}
		The $\mathcal{C}(\phi)$-radii for the functions $f_{u}, g_{u}$ and $h_{u}$ are the smallest positive roots of the following equations respectively:
		\begin{enumerate}[$(i)$]
			\item $ {rf''_{u}(r)}+\alpha {f'_{u}(r)}=0$;
			
			\item ${rg''_{u}(r)}+\alpha {g'_{u}(r)}=0$;
			
			\item ${rh''_{u}(r)}+\alpha {h'_{u}(r)}=0$,
			
		\end{enumerate}
		where $\alpha$ is the radius of the largest disk $\{w: |w-1|< \alpha\}$ inside $\phi(\mathbb{D})$.	
	\end{theorem}
	\begin{proof}
		We begin with the first part. From \eqref{fL}, we have
		\begin{equation}\label{f1}
		1+\dfrac{zf''_{u}(z)}{f'_{u}(z)}= 1+\dfrac{z\mathcal{L}''_{u}(z)}{\mathcal{L}'_{u}(z)} +\left(\dfrac{1}{u+\dfrac{1}{2}}-1 \right)\dfrac{z\mathcal{L}'_{u}(z)}{\mathcal{L}_{u}(z)}.
		\end{equation}
		Also using the result \cite[Lemma~1]{Bricz-Rama}, we have
		\begin{equation*}
		\mathcal{L}_{u}(z)=\dfrac{z^{u+\dfrac{1}{2}}}{u(u+1)}\Phi_0(z)=\dfrac{z^{u+\dfrac{1}{2}}}{u(u+1)} \prod_{n\geq1}\left(1-\dfrac{z^2}{\tau^2_{u,n}}\right),
		\end{equation*}
		where $\Phi_k(z):={}_1F_{2}\left(1; \dfrac{u-k+2}{2}, \dfrac{u-k+3}{2};-\dfrac{z^2}{4}\right)$ with conditions as mentioned in \cite[Lemma~1]{Bricz-Rama}, and from the proof of \cite[Theorem~3]{Bricz-Rama}, we see that the entire function $\dfrac{u(u+1)}{u+\dfrac{1}{2}} z^{-u+\dfrac{1}{2}}\mathcal{L}'_{u}(z)$ is of order $1/2$ and thus has the following Hadamard factorization:
		\begin{equation*}
		\mathcal{L}'_{u}(z)=\dfrac{u+\dfrac{1}{2}}{u(u+1)}z^{u-\dfrac{1}{2}} \prod_{n\geq 1}\left(1-\dfrac{z^2}{\breve{\tau}^2_{u,n}} \right),
		\end{equation*}
		where $\tau_{u,n}$ and $\breve{\tau}_{u,n}$ are the $n$-th positive zeros of $\mathcal{L}_{u}$ and $\mathcal{L}'_{u}$, respectively and interlace for $0\neq u\in(-1,1)$ (see \cite[Theorem~1]{Bricz-Rama}). Now we can rewrite \eqref{f1} as follows:
		\begin{equation*}
		1+\dfrac{zf''_{u}(z)}{f'_{u}(z)}=1-\left(\dfrac{1}{u+\dfrac{1}{2}}-1\right)\sum_{n\geq 1}\dfrac{2z^2}{\tau^2_{u,n}-z^2}-\sum_{n\geq 1}\dfrac{2z^2}{\breve{\tau}^2_{u,n}-z^2}.
		\end{equation*}
		Let us now consider the case $u\in(0,1/2]$. Then using the inequality $||x|-|y||\leq |x-y|$ for $|z|=r<\breve{\tau}_{u,1}<\tau_{u,1}$ we get
		\begin{equation}\label{final-f}
		\left|\dfrac{zf''_{u}(z)}{f'_{u}(z)}\right|\leq \left(\dfrac{1}{u+\dfrac{1}{2}}-1\right)\sum_{n\geq 1}\dfrac{2r^2}{\tau^2_{u,n}-r^2}+\sum_{n\geq 1}\dfrac{2r^2}{\breve{\tau}^2_{u,n}-r^2} =-\dfrac{rf''_{u}(r)}{f'_{u}(r)}
		\end{equation}
		and for the case $u\in(1/2,1)$, using the inequality \eqref{firstnorm} with $\lambda=1-1/(u+1/2)$, we also get
		\begin{equation}\label{final-fu}
		\left|\dfrac{zf''_{u}(z)}{f'_{u}(z)}\right|\leq -\dfrac{rf''_{u}(r)}{f'_{u}(r)},
		\end{equation}	
		which is same as \eqref{final-f}. When $u\in (-1,0)$, then we proceed similarly substituting $u$ by $u-1$, $\Phi_0$ by $\Phi_1$, where $\Phi_1$ belongs to the Laguerre-P\'{o}lya class $\mathcal{LP}$ and the $n$-th positive zeros $\xi_{u,n}$ and $\breve{\xi}_{u,n}$ of $\Phi_1$ and its derivative $\Phi'_1$, respectively interlace. Finally, replacing $u$ by $u+1$, we obtain the required inequality.\\
		\indent For $0\neq u\in (-1,1)$, the Hadamard factorization for the entire functions $g'_{u}$ and $h'_{u}$ of order $1/2$ \cite[Theorem~3]{Bricz-Rama} is given by 
		\begin{equation}\label{Had-gh}
		g'_{u}(z)=\prod_{n\geq 1}\left( 1-\dfrac{z^2}{\gamma^2_{u,n}}\right) \quad \text{and} \quad
		h'_{u}(z)=\prod_{n\geq 1}\left( 1-\dfrac{z}{\delta^2_{u,n}}\right),
		\end{equation}
		where $\gamma_{u,n}$ and $\delta_{u,n}$ are $n$-th positive zeros of $g'_{u}$ and $h'_{u}$, respectively and $\gamma_{u,1}, \delta_{u,1} < \tau_{u,1}$. Now from \eqref{fL} and \eqref{Had-gh}, we have
		\begin{align}\label{gh-u}
		1+\dfrac{zg''_{u}(z)}{g'_{u}(z)}&= \dfrac{1}{2}-u +z\dfrac{(\dfrac{3}{2}-u)\mathcal{L}'_{u}(z)+z\mathcal{L}''_{u}(z)}{(\dfrac{1}{2}-u)\mathcal{L}_{u}(z)+z\mathcal{L}'_{u}(z)} =1-\sum_{n\geq 1}\dfrac{2z^2}{\gamma^2_{u,n}-z^2}  \nonumber\\
		1+\dfrac{zh''_{u}(z)}{h'_{u}(z)}&= \dfrac{1}{2}\left(\dfrac{3}{2}-u +\sqrt{z}\dfrac{(\dfrac{5}{2}-u)\mathcal{L}'_{u}(\sqrt{z})+\sqrt{z}\mathcal{L}''_{u}(\sqrt{z})}{(\dfrac{3}{2}-u)\mathcal{L}_{u}(\sqrt{z})+\sqrt{z}\mathcal{L}'_{u}(\sqrt{z})}    \right)=1-\sum_{n\geq 1}\dfrac{z}{\delta^2_{u,n}-z}.
		\end{align}
		Using the inequality $||x|-|y||\leq |x-y|$ in \eqref{gh-u} for $|z|=r< \gamma_{u,1}$ and $|z|=r<\delta_{u,1}$, we get
		\begin{equation}\label{final-gh}
		\left|\dfrac{zg''_{u}(z)}{g'_{u}(z)} \right|\leq \sum_{n\geq 1}\dfrac{2r^2}{\gamma^2_{u,n}-r^2}=-\dfrac{rg''_{u}(r)}{g'_{u}(r)}  \quad \text{and} \quad 
		\left|\dfrac{zh''_{u}(z)}{h'_{u}(z)} \right|\leq \sum_{n\geq 1}\dfrac{r}{\delta^2_{u,n}-r}=-\dfrac{rh''_{u}(r)}{h'_{u}(r)}. 
		\end{equation} 
		Now let $\alpha$ be the largest such that $\{w: |w-1|\leq \alpha\} \subseteq \phi(\mathbb{D})$. Then from \eqref{final-f}, \eqref{final-fu} and \eqref{final-gh}, we see that $f_{u}, g_{u}$ and $h_{u}$ belong to $\mathcal{C}(1+\alpha z)\subseteq \mathcal{C}(\phi)$, whenever the following inequalities
		\begin{align*}
		-\dfrac{rf''_{u}(r)}{f'_{u}(r)}\leq \alpha, \quad -\dfrac{rg''_{u}(r)}{g'_{u}(r)}\leq \alpha \quad \text{and}\quad -\dfrac{rh''_{u}(r)}{h'_{u}(r)}\leq\alpha
		\end{align*}
		hold. Further proceeding as in Theorem~\ref{wright-star}, we obtain the desired equations. Sharpness of the radii follow with the suitable rotations of the functions $f_{u}, g_{u}$ and $h_{u}$. \qed 
	\end{proof}

	\subsection{ Convexity of Struve functions}
	The Struve function $\mathcal{\bf{H}}_{\beta}$ of first kind is a particular solution of the second-order inhomogeneous Bessel differential equation $$z^2w''(z)+zw'(z)+(z^2-{\beta}^2)w(z)=\dfrac{4(\dfrac{z}{2})^{\beta+1}}{\sqrt{\pi}\Gamma(\beta+\dfrac{1}{2})}$$ and have the following form:
	\begin{equation*}
	\mathcal{\bf{H}}_{\beta}(z):=\dfrac{(\dfrac{z}{2})^{\beta+1}}{\sqrt{\dfrac{\pi}{4}}\Gamma(\beta+\dfrac{1}{2})} {}_1 F_{2}\left(1;\dfrac{3}{2},\beta+\dfrac{3}{2};-\dfrac{z^2}{4}\right) ,
	\end{equation*}
	where $-\beta-\dfrac{3}{2}\notin\mathbb{N}$ and ${}_1 F_{2}$ is a hypergeometric function. Since it is not normalized, so we consider the following three normalized functions involving $\mathcal{\bf{H}}_{\beta}$ :
	\begin{align}\label{nor-strv-uvw}
	U_{\beta}(z)&=\left(\sqrt{\pi}2^{\beta}(\beta+\dfrac{3}{2}){\bf{H}}_{\beta}(z)\right)^{\dfrac{1}{\beta+1}},\nonumber\\
	V_{\beta}(z)&=\sqrt{\pi}2^{\beta}z^{-\beta}\Gamma(\beta+\dfrac{3}{2}){\bf{H}}_{\beta}(z),\nonumber\\
	W_{\beta}(z)&=\sqrt{\pi}2^{\beta}z^{\dfrac{1-\beta}{2}}\Gamma(\beta+\dfrac{3}{2}){\bf{H}}_{\beta}(\sqrt{z}).
	\end{align}
	Moreover, for $|\beta|\leq\dfrac{1}{2}$, ${\bf{H}}_{\beta}$ and ${\bf{H}}'_{\beta}$ have the Hadamard factorizations \cite[Theorem~4]{Bricz-Rama} given by
	\begin{align}\label{strv-facto}
	{\bf{H}}_{\beta}(z)=\dfrac{z^{\beta+1}}{\sqrt{\pi}2^{\beta}\Gamma(\beta+\dfrac{3}{2})}\prod_{n\geq1}\left(1-\dfrac{z^2}{z^2_{\beta,n}}\right) \quad \text{and} \quad
	{\bf{H}}'_{\beta}(z)=\dfrac{(\beta+1)z^{\beta}}{\sqrt{\pi}2^{\beta}\Gamma(\beta+\dfrac{3}{2})}\prod_{n\geq1}\left(1-\dfrac{z^2}{\breve{z}^2_{\beta,n}}\right)
	\end{align}
	where $z_{\beta,n}$ and $\breve{z}_{\beta,n}$ are the $n$-th positive zeros of ${\bf{H}}_{\beta}$ and ${\bf{H}}'_{\beta}$,respectively and interlace \cite[Theorem~2]{Bricz-Rama}. Thus from \eqref{strv-facto} with logarithmic differentiation, we obtain respectively
	\begin{equation}\label{strv-strlikeconvex}
	\dfrac{z{\bf{H}}'_{\beta}(z)}{{\bf{H}}_{\beta}(z)}=(\beta+1)-\sum_{n\geq1}\dfrac{2z^2}{z^2_{\beta,n}-z^2} \quad \text{and} \quad 
	1+\dfrac{z{\bf{H}}''_{\beta}(z)}{{\bf{H}}'_{\beta}(z)}=(\beta+1)-\sum_{n\geq1}\dfrac{2z^2}{\breve{z}^2_{\beta,n}-z^2} .
	\end{equation}
	Also for $|\beta|\leq \dfrac{1}{2}$, the Hadamard factorization for the entire functions $V'_{\beta}$ and $W'_{\beta}$ of order $1/2$ \cite[Theorem~4]{Bricz-Rama} is given by 
	\begin{equation}\label{Had-VW'}
	V'_{\beta}(z)=\prod_{n\geq 1}\left( 1-\dfrac{z^2}{\eta^2_{\beta,n}}\right) \quad \text{and} \quad
	W'_{\beta}(z)=\prod_{n\geq 1}\left( 1-\dfrac{z}{\sigma^2_{\beta,n}}\right),
	\end{equation}
	where $\eta_{\beta,n}$ and $\sigma_{\beta,n}$ are $n$-th positive zeros of $V'_{\beta}$ and $W'_{\beta}$, respectively. $V'_{\beta}$ and $W'_{\beta}$ belong to the Laguerre-P\'{o}lya class and zeros satisfy $\eta_{\beta,1}, \sigma_{\beta,1} < z_{\beta,1}$. Now proceeding as in Theorem~\ref{lommeltheorem} using \eqref{nor-strv-uvw}, \eqref{strv-facto}, \eqref{strv-strlikeconvex} and \eqref{Had-VW'}, we obtain the following result:
	
	\begin{theorem}\label{struvetheorem}
		Let $|\beta|\leq 1/2$. Then $\mathcal{C}(\phi)$-radii for the functions $U_{\beta}, V_{\beta}$ and $W_{\beta}$ are the smallest positive roots of the following equations respectively:
		\begin{enumerate}[$(i)$]
			\item $ {rU''_{\beta}(r)}+\alpha {U'_{\beta}(r)}=0$;
			
			\item $ {rV''_{\beta}(r)}+\alpha {V'_{\beta}(r)}=0$;
			
			\item $ {rW''_{\beta}(r)}+\alpha {W'_{\beta}(r)}=0$,
			
		\end{enumerate}
		where $\alpha$ is the radius of the largest disk $\{w: |w-1|< \alpha\}$ inside $\phi(\mathbb{D})$.	
	\end{theorem}

	\section{On Ramanujan type entire functions}
	Ismail and Zhang~\cite{ismail-2018} defined the following entire function of growth order zero for $\beta>0$, called Ramanujan type entire function
	\begin{equation*}
	A^{(\beta)}_{q}(a,z)= \sum_{n\geq0}\dfrac{(a;q)_n q^{\beta n^2}}{(q;q)_n}z^n,
	\end{equation*}
	where $\beta>0$, $0<q<1$, $a\in \mathbb{C}$, $(a;q)_0=1$ and $(a;q)_k=\prod_{j=0}^{k-1}(1-aq^j)$ for $k\geq1,$ which is the generalization of both the Ramanujan entire function $A_{q}(z)$ and Stieltjes-Wigert polynomial $S_n(z;q)$.
	Since $A^{(\beta)}_{q}(a,z)\not \in \mathcal{A}$, therefore consider the following three normalized functions in $\mathcal{A}$:
	\begin{align}\label{ramj1}
	f_{\beta, q}(a,z)&= \left(z^{\beta} A^{(\beta)}_{q}(-a,-z^2)  \right)^{1/\beta} \nonumber\\
	g_{\beta, q}(a,z)&= z A^{(\beta)}_{q}(-a,-z^2) \nonumber\\
	h_{\beta, q}(a,z)&= z A^{(\beta)}_{q}(-a,-z),
	\end{align}
	where $\beta>0$, $a\geq0$ and $0<q<1$. From \cite[Lemma~2.1, p.~4-5]{ErhanDenij2020}, we see that the function
	$$z\rightarrow \Psi_{\beta,q}(a,z):= A^{(\beta)}_{q}(-a,-z^2)$$
	has infinitely many zeros (all are positive) for $\beta>0$, $a\geq0$ and $0<q<1$. Let $\psi_{\beta,q,n}(a)$ be the $n$-th positive zero of $\Psi_{\beta,q}(a,z)$. Then it has the following Weiersstrass decomposition:
	\begin{equation}\label{ramj2}
	\Psi_{\beta,q}(a,z)= \prod_{n\geq1}\left(1-\dfrac{z^2}{\psi^2_{\beta, q, n}(a)} \right).
	\end{equation}
	Moreover, the $n$-th positive zero $\Xi_{\beta, q, n}(a)$ of the derivative of the  following function
	\begin{equation}\label{ramj3}
	\Phi_{\beta, q}(a,z):= z^{\beta}\Psi_{\beta, q}(a,z)
	\end{equation}
	interlace with $\psi_{\beta, q, n}(a)$ and satisfy the relation
	$\Xi_{\beta, q, n}(a)< \psi_{\beta,q,n}(a)<\Xi_{\beta, q, n+1}(a)< \psi_{\beta,q,n+1}(a)$ for $n\geq1.$ Now using \eqref{ramj1} and \eqref{ramj2}, we have
	\begin{align*}\label{ramj-star}
	\dfrac{zf'_{\beta,q}(a,z)}{f_{\beta,q}(a,z)}&= 1+ \dfrac{1}{\beta}\dfrac{z\Psi'_{\beta,q}(a,z)}{\Psi_{\beta,q}(a,z)}
	=1-\dfrac{1}{\beta}\sum_{n\geq1}\dfrac{2z^2}{\psi^2_{\beta,q,n}(a)- z^2} ; \;(a>0) \nonumber\\
	\dfrac{zg'_{\beta,q}(a,z)}{g_{\beta,q}(a,z)}&= 1+ \dfrac{z\Psi'_{\beta,q}(a,z)}{\Psi_{\beta,q}(a,z)}
	=1-\sum_{n\geq1}\dfrac{2z^2}{\psi^2_{\beta,q,n}(a)- z^2} ; \nonumber\\
	\dfrac{zh'_{\beta,q}(a,z)}{h_{\beta,q}(a,z)}&= 1+ \dfrac{1}{2}\dfrac{\sqrt{z} \Psi'_{\beta,q}(a,\sqrt{z})}{\Psi_{\beta,q}(a,\sqrt{z})}
	=1-\sum_{n\geq1}\dfrac{z}{\psi^2_{\beta,q,n}(a)- z},
	\end{align*}
	where $\beta>0,a\geq0$ and $0<q<1$. Also, using \eqref{ramj3} and the infinite product representation of $\Phi'$~\cite[p.~14-15, Also see Eq.~4.6]{ErhanDenij2020}, we have
	\begin{align*}
	1+\dfrac{zf''_{\beta,q}(a,z)}{f'_{\beta,q}(a,z)}&= 1+\dfrac{z \Phi''_{\beta,q}(a,z)}{\Phi'_{\beta,q}(a,z)} 
	+\left(\dfrac{1}{\beta}-1\right) \dfrac{z \Phi'_{\beta,q}(a,z)}{\Phi_{\beta,q}(a,z)}\\
	&=1-\sum_{n\geq1}\dfrac{2z^2}{\Xi^2_{\beta,q,n}(a)- z^2} -\left(\dfrac{1}{\beta}-1\right) \sum_{n\geq1}\dfrac{2z^2}{\psi^2_{\beta,q,n}(a)- z^2}.
	\end{align*}
	As $(z\Psi_{\beta,q}(a,z))'$ and $h'_{\beta,q}(a,z)$ belongs to $\mathcal{LP}$. So suppose $\gamma_{\beta,q,n}(a)$ be the positive zeros of $g'_{\beta,q}(a,z)$ (growth order is same as $\Psi_{\beta,q}(a,z)$) and $\delta_{\beta,q,n}(a)$ be the positive zeros of $h'_{\beta,q}(a,z)$. Thus using their infinite product representations, we have
	\begin{align*}
	1+\dfrac{zg''_{\beta,q}(a,z)}{g'_{\beta,q}(a,z)}&= 1-\sum_{n\geq1}\dfrac{2z^2}{\gamma^2_{\beta,q,n}(a)- z^2} \nonumber\\
	1+\dfrac{zh''_{\beta,q}(a,z)}{h'_{\beta,q}(a,z)}&= 1-\sum_{n\geq1}\dfrac{z}{\delta^2_{\beta,q,n}(a)- z}. 
	\end{align*}
	
	Now proceeding similarly as done in the above sections, we obtain the following results:
	\begin{theorem}\label{ramjThm1}
		Let $\beta>0$, $a\geq0$ and $0<q<1$. Then $\mathcal{S}^*(\phi)$-radii for the functions $f_{\beta, q}(a,z)$, $g_{\beta, q}(a,z)$ and $h_{\beta, q}(a,z)$ are the smallest positive roots of the following equations respectively:
		\begin{enumerate}[$(i)$]
			\item $r \Psi'_{\beta,q}(a,r)+\beta \alpha \Psi_{\beta,q}(a,r)=0$;
			
			\item $r \Psi'_{\beta,q}(a,r)+ \alpha \Psi_{\beta,q}(a,z)=0$;
			
			\item $\sqrt{r} \Psi'_{\beta,q}(a,\sqrt{r})+ 2\alpha \Psi_{\beta,q}(a,\sqrt{r})=0$,
			
		\end{enumerate}
		where $\alpha$ is the radius of the largest disk $\{w: |w-1|< \alpha\}$ inside $\phi(\mathbb{D})$.
	\end{theorem}
	
	\begin{theorem}\label{ramjThm2}
		Let $\beta>0$, $a\geq0$ and $0<q<1$. Then $\mathcal{C}(\phi)$-radii for the functions $f_{\beta, q}(a,z)$, $g_{\beta, q}(a,z)$ and $h_{\beta, q}(a,z)$ are the smallest positive roots of the following equations respectively:
		\begin{enumerate}[$(i)$]
			\item $ \dfrac{r \Phi''_{\beta,q}(a,r)}{\Phi'_{\beta,q}(a,r)}+\left(\dfrac{1}{\beta}-1\right) \dfrac{r \Phi'_{\beta,q}(a,r)}{\Phi_{\beta,q}(a,r)}+\alpha=0$;
			
			\item ${rg''_{\beta, q}(a,r)}+\alpha {g'_{\beta, q}(a,r)}=0$;
			
			\item ${rh''_{\beta, q}(a,r)}+\alpha {h'_{\beta, q}(a,r)}=0$,
			
		\end{enumerate}
		where $\alpha$ is the radius of the largest disk $\{w: |w-1|< \alpha\}$ inside $\phi(\mathbb{D})$.
	\end{theorem}

	\section{ Some Applications and Further Results}
	\subsection{Applications}
	In the following result, we consider the Carathe\'{o}dory functions $\phi$ associated with some well known classes as well as some recently introduced in \cite{sinefun,pap-sokol1996,mendi2exp,mendi,kanas,janow,goel2020}:
		\begin{corollary}\label{application}
		If $\alpha$ be the radius of the largest disk $\{w: |w-1|< \alpha\}$ inside $\phi(\mathbb{D})$, where
		\begin{enumerate}[$(i)$]
			\item   	$\alpha=\min\left\{\left|1-\dfrac{1+D}{1+E}\right|, \left|1-\dfrac{1-D}{1-E}\right|\right\}=\dfrac{D-E}{1+|E|}$ when $\phi(z)= \dfrac{1+Dz}{1+Ez}$, where $-1\leq E<D\leq1$;
			
			\item  	$\alpha=\sqrt{2-2\sqrt{2}+\sqrt{-2+2\sqrt{2}}}$ when  $\phi(z)=\sqrt{2}-(\sqrt{2}-1)\sqrt{\dfrac{1-z}{1+2(\sqrt{2}-1)z}}$;
			
			\item  	$\alpha=\sqrt{2}-1$ when $\phi(z)=\sqrt{1+z}$;
			
			\item  $\alpha=e-1$ when $\phi(z)=e^z$;
			
			\item  $\alpha=2-\sqrt{2}$ when $\phi(z)=z+\sqrt{1+z^2}$;
			
			\item  $\alpha=1/e$ when $\phi(z)=1+ze^z$;
			
			\item  $\alpha=\dfrac{e-1}{e+1}$ when $\phi(z)=\dfrac{2}{1+e^{-z}}$;
			
			\item   $\alpha=\sin{1}$ when $\phi(z)=1+\sin{z}$;
			
			\item  for the domains bounded by the conic sections
			$\Omega_\kappa:=\{w=u+iv: u^2>{\kappa}^2(u-1)^2+{\kappa}^2v^2; \kappa\in[0,\infty)  \},$ we have $$\alpha=\dfrac{1}{\kappa+1},$$ 
			where the boundary curve of $\Omega_\kappa$ for fixed $\kappa$ is represented by the imaginary axis $(\kappa=0)$, the right branch of a hyperbola $(0<\kappa<1)$, a parabola $(\kappa=1)$ and an ellipse $(\kappa>1)$. The univalent Carathe\'{o}dory functions mapping $\mathbb{D}$ onto $\Omega_\kappa$ is given by
			\begin{equation*}
			\phi(z):=\phi_{\kappa}(z)= 
			\left\{
			\begin{array}{lll}
			\dfrac{1+z}{1-z} & $for$ & \kappa=0;\\
			1+\dfrac{2}{1-\kappa^2}\sinh^2(A(\kappa) arctanh\sqrt{z}) & $for$ & \kappa\in(0,1);\\
			1+\dfrac{2}{\pi^2}\log^2{\dfrac{1+\sqrt{z}}{1-\sqrt{z}}} & $for$ & \kappa=1;\\
			1+\dfrac{2}{\kappa^2-1}\sin^2\left(\dfrac{\pi}{2K(t)}F\left( \dfrac{\sqrt{z}}{\sqrt{t}}, t \right) \right) & $for$ & \kappa>1,		
			\end{array}	
			\right.
			\end{equation*}
			where $A(\kappa)=(2/\pi)\arccos(\kappa)$, $F(w,t)=\int_{0}^{w}\dfrac{dx}{\sqrt{(1-x^2)(1-t^2x^2)}}$ is the Legender elliptic integral of the first kind, $K(t)=F(1,t)$ and $t\in(0,1)$ is choosen such that $\kappa=\cosh(\pi K'(t)/2K(t))$.
		\end{enumerate}
		Then Theorems~\ref{wright-phi},~\ref{wright-c},~\ref{mittag-phi},~\ref{mittag-c},~\ref{Leg-c},~\ref{lommeltheorem},~\ref{struvetheorem}, ~\ref{ramjThm1} and~\ref{ramjThm2} hold true for the above choices of $\phi$ respectively.
	\end{corollary}
	In the above corollary for the janowski functions at $(i)$, we use its inverse representaion $\left|({w-1})/({D-Ew}) \right|<1$ for the sharpness (also see \cite{vibha}). Whereas for the Lemniscate of Bernoulli at $(iii)$, we use the fact that if $|w-1|\leq \sqrt{2}-1$, then $|w+1|\leq \sqrt{2}+1$, which implies $|w^2-1|\leq1$.

	\subsection{Radius of Strongly Starlikness}
		To prove our next result, we need the following lemma:
	\begin{lemma}\cite{ganga1997}\label{Ravi}
		If $|z|\leq r<1$ and $|z_k|=R>r$, then we have
		$$\left|\dfrac{z}{z-z_k}+\dfrac{r^2}{R^2-r^2} \right|\leq \dfrac{Rr}{R^2-r^2}.$$
	\end{lemma}	
	\begin{theorem}[Wright functions]\label{W-strgradius}
		Let $\rho, \beta>0$. Then $\mathcal{S}^{*}\left(\bigg(\dfrac{1+z}{1-z}\bigg)^\epsilon \right)$-radii for the functions $f_{\rho,\beta}$, $g_{\rho,\beta}$ and $h_{\rho,\beta}$ are the unique positive roots of the following equations:
		\begin{enumerate}[$(i)$]
			\item $\dfrac{2}{\beta} \sum_{n\geq1}\left( \dfrac{{\zeta}^2_{\rho,\beta,n} r^2}{{\zeta}^4_{\rho,\beta,n}-r^4}+ \sin\left( \dfrac{\pi \epsilon}{2}\right) \dfrac{r^4}{{\zeta}^4_{\rho,\beta,n}-r^4} \right) - \sin\left(\dfrac{\pi \epsilon}{2}\right) =0$;
			
			\item  $2 \sum_{n\geq1}\left( \dfrac{{\zeta}^2_{\rho,\beta,n} r^2}{{\zeta}^4_{\rho,\beta,n}-r^4}+ \sin\left( \dfrac{\pi \epsilon}{2}\right) \dfrac{r^4}{{\zeta}^4_{\rho,\beta,n}-r^4} \right) - \sin\left(\dfrac{\pi \epsilon}{2}\right) =0$;
			
			\item $\sum_{n\geq1}\left( \dfrac{{\zeta}^2_{\rho,\beta,n} r}{{\zeta}^4_{\rho,\beta,n}-r^2}+ \sin\left( \dfrac{\pi \epsilon}{2}\right) \dfrac{r^2}{{\zeta}^4_{\rho,\beta,n}-r^2} \right) - \sin\left(\dfrac{\pi \epsilon}{2}\right) =0$
			
		\end{enumerate}
		in $(0,\zeta_{\rho,\beta,1})$, $(0,\zeta_{\rho,\beta,1})$ and $(0,\zeta^2_{\rho,\beta,1})$ respectively.
	\end{theorem}
	\begin{proof}
		We prove the first part and the rest all follow in the similar manner. From \eqref{w-sharp} and using Lemma~\ref{Ravi}, we see that
		\begin{align*}
		\dfrac{zf'_{\rho,\beta}(z)}{f_{\rho,\beta}(z)} = 1+ \dfrac{1}{\beta} \dfrac{z W'_{\rho, \beta}(z)}{W_{\rho, \beta}(z)}=1-\dfrac{1}{\beta} \sum_{n\geq1}\dfrac{2z^2}{{\zeta}^2_{\rho,\beta,n} -z^2},
		\end{align*}
		which implies
		\begin{align}\label{strg-Disk}
		\left| 
		\dfrac{zf'_{\rho,\beta}(z)}{f_{\rho,\beta}(z)} -\left(1- \dfrac{1}{\beta} \sum_{n\geq1}\dfrac{2r^4}{{\zeta}^4_{\rho,\beta,n} -r^4} \right)  \right| 
		&\leq\dfrac{1}{\beta} \sum_{n\geq1}\left| \dfrac{2z^2}{z^2 - {\zeta}^2_{\rho,\beta,n}}+ \dfrac{2r^4}{{\zeta}^4_{\rho,\beta,n}} -r^4 \right| \nonumber\\
		&\leq \dfrac{2}{\beta} \sum_{n\geq1}\dfrac{{\zeta}^2_{\rho,\beta,n} r^2}{{\zeta}^4_{\rho,\beta,n} -r^4} 
		\end{align}
		for $|z|\leq r<{\zeta}_{\rho,\beta,1}.$ Define
		\begin{equation*}
		a:= \left(1- \dfrac{1}{\beta} \sum_{n\geq1}\dfrac{2r^4}{{\zeta}^4_{\rho,\beta,n} -r^4} \right) \quad\text{and}\quad R_{a}:= \dfrac{2}{\beta} \sum_{n\geq1}\dfrac{{\zeta}^2_{\rho,\beta,n} r^2}{{\zeta}^4_{\rho,\beta,n} -r^4}.  
		\end{equation*}
		Now from Lemma~\cite[Lemma~3.1, p.~307]{ganga1997}, we see that the disk $|w-a|\leq R_a$ in \eqref{strg-Disk} is contained in the sector $|\arg{w}|\leq {\pi \epsilon}/{2}$, whenever
		\begin{equation} \label{eqbySlemma}
		\dfrac{2}{\beta} \sum_{n\geq1}\dfrac{{\zeta}^2_{\rho,\beta,n} r^2}{{\zeta}^4_{\rho,\beta,n} -r^4}   \leq 
		\left(\left(1- \dfrac{1}{\beta} \sum_{n\geq1}\dfrac{2r^4}{{\zeta}^4_{\rho,\beta,n} -r^4} \right) \right) \sin\left(\dfrac{\pi\epsilon}{2}\right)
		\end{equation}
		holds. Let us now define
		\begin{equation*}
		T(r):= \dfrac{2}{\beta}\sum_{n\geq1}^{} \left( \dfrac{{\zeta}^2_{\rho,\beta,n} r^2}{{\zeta}^4_{\rho,\beta,n}-r^4}+ \sin\left( \dfrac{\pi \epsilon}{2}\right) \dfrac{r^4}{{\zeta}^4_{\rho,\beta,n}-r^4} \right) - \sin\left(\dfrac{\pi \epsilon}{2}\right).
		\end{equation*}
		Then a simple calculation shows that $T'(r)\geq0$. Moreover, $\lim_{r\rightarrow0}T(r)<0$ and $\lim_{r\rightarrow {\zeta}_{\rho,\beta,1}}T(r)>0$. Thus \eqref{eqbySlemma} holds in $|z|\leq r_0$, where $r_0$ is the unique positive root of $T(r)=0$ in $(0, {\zeta}_{\rho,\beta,1})$. This completes the proof. \qed
	\end{proof}
	
		The following results can be derived in a similar fashion as dealt in Theorem~\ref{W-strgradius}. So the proofs are omitted here.
		\begin{theorem}[On Lommel functions]
		The $\mathcal{S}^{*}\left((\dfrac{1+z}{1-z})^\epsilon \right)$-radii for the functions $f_{u}, g_{u}$ and $h_{u}$ are the unique positive roots of the following equations:
		\begin{enumerate}[$(i)$]
			\item $\dfrac{2}{u+\dfrac{1}{2}} \sum_{n\geq1}\left( \dfrac{{\tau}^2_{u,n} r^2}{{\tau}^4_{u,n}-r^4}+ \sin\left( \dfrac{\pi \epsilon}{2}\right) \dfrac{r^4}{{\tau}^4_{u,n}-r^4} \right) - \sin\left(\dfrac{\pi \epsilon}{2}\right) =0$;
			
			\item  $2 \sum_{n\geq1}\left( \dfrac{{\tau}^2_{u,n} r^2}{{\tau}^4_{u,n}-r^4}+ \sin\left( \dfrac{\pi \epsilon}{2}\right) \dfrac{r^4}{{\tau}^4_{u,n}-r^4} \right) - \sin\left(\dfrac{\pi \epsilon}{2}\right) =0$;
			
			\item $\sum_{n\geq1}\left( \dfrac{{\tau}^2_{u,n} r}{{\tau}^4_{u,n}-r^2}+ \sin\left( \dfrac{\pi \epsilon}{2}\right) \dfrac{r^2}{{\tau}^4_{u,n}-r^2} \right) - \sin\left(\dfrac{\pi \epsilon}{2}\right) =0$
			
		\end{enumerate}	
		in $(0,\tau_{u,1})$, $(0,\tau_{u,1})$ and $(0,\tau^2_{u,1})$  respectively.
	\end{theorem}

	\begin{theorem}[Struve functions]
		Let $|\beta|\leq 1/2$. Then $\mathcal{S}^{*}\left((\dfrac{1+z}{1-z})^\epsilon \right)$-radii for the functions $U_{\beta}, V_{\beta}$ and $W_{\beta}$ are the unique positive roots of the following equations:
		\begin{enumerate}[$(i)$]
			\item $\dfrac{2}{\beta+1} \sum_{n\geq1}\left( \dfrac{{z}^2_{\beta,n} r^2}{{z}^4_{\beta,n}-r^4}+ \sin\left( \dfrac{\pi \epsilon}{2}\right) \dfrac{r^4}{{z}^4_{\beta,n}-r^4} \right) - \sin\left(\dfrac{\pi \epsilon}{2}\right) =0$;
			
			\item  $2 \sum_{n\geq1}\left( \dfrac{{z}^2_{\beta,n} r^2}{{z}^4_{\beta,n}-r^4}+ \sin\left( \dfrac{\pi \epsilon}{2}\right) \dfrac{r^4}{{z}^4_{\beta,n}-r^4} \right) - \sin\left(\dfrac{\pi \epsilon}{2}\right) =0$;
			
			\item $\sum_{n\geq1}\left( \dfrac{{z}^2_{\beta,n} r}{{z}^4_{\beta,n}-r^2}+ \sin\left( \dfrac{\pi \epsilon}{2}\right) \dfrac{r^2}{{z}^4_{\beta,n}-r^2} \right) - \sin\left(\dfrac{\pi \epsilon}{2}\right) =0$
			
		\end{enumerate}
		in $(0,z_{\beta,1})$, $(0,z_{\beta,1})$ and $(0,z^2_{\beta,1})$ respectively.
	\end{theorem}
	
	\begin{theorem}[Mittag-Leffler functions]
		Let $(\tfrac{1}{\mu},\nu)\in W_{i}$ and $a>0$. Then $\mathcal{S}^{*}\left((\dfrac{1+z}{1-z})^\epsilon \right)$-radii  for the functions $f_{\mu,\nu, a}$, $g_{\mu,\nu, a}$ and $h_{\mu,\nu, a}$ are the unique positive roots of the following equations:
		\begin{enumerate}[$(i)$]
			\item $\dfrac{2}{\nu} \sum_{n\geq1}\left( \dfrac{\lambda^2_{\mu,\nu,a,n} r^2}{\lambda^4_{\mu,\nu,a,n}-r^4}+ \sin\left( \dfrac{\pi \epsilon}{2}\right) \dfrac{r^4}{\lambda^4_{\mu,\nu,a,n}-r^4} \right) - \sin\left(\dfrac{\pi \epsilon}{2}\right) =0$;
			
			\item  $2 \sum_{n\geq1}\left( \dfrac{\lambda^2_{\mu,\nu,a,n} r^2}{\lambda^4_{\mu,\nu,a,n}-r^4}+ \sin\left( \dfrac{\pi \epsilon}{2}\right) \dfrac{r^4}{\lambda^4_{\mu,\nu,a,n}-r^4} \right) - \sin\left(\dfrac{\pi \epsilon}{2}\right) =0$;
			
			\item $\sum_{n\geq1}\left( \dfrac{\lambda^2_{\mu,\nu,a,n} r}{\lambda^4_{\mu,\nu,a,n}-r^2}+ \sin\left( \dfrac{\pi \epsilon}{2}\right) \dfrac{r^2}{\lambda^4_{\mu,\nu,a,n}-r^2} \right) - \sin\left(\dfrac{\pi \epsilon}{2}\right) =0$
			
		\end{enumerate}
		in $(0,\lambda_{\mu,\nu,a,1})$, $(0,\lambda_{\mu,\nu,a,1})$ and $(0,\lambda^2_{\mu,\nu,a,1})$  respectively.
	\end{theorem}
	
	\begin{theorem}[Ramanujan type entire functions]
		Let $\beta>0$, $a\geq0$ and $0<q<1$. Then $\mathcal{S}^{*}\left((\dfrac{1+z}{1-z})^\epsilon \right)$-radii for the functions $f_{\beta, q}(a,z)$, $g_{\beta, q}(a,z)$ and $h_{\beta, q}(a,z)$ are the unique positive roots of the following equations:
		\begin{enumerate}[$(i)$]
			\item $\dfrac{2}{\beta} \sum_{n\geq1}\left( \dfrac{\psi^2_{\beta,q,n}(a) r^2}{\psi^4_{\beta,q,n}(a)-r^4}+ \sin\left( \dfrac{\pi \epsilon}{2}\right) \dfrac{r^4}{\psi^4_{\beta,q,n}(a)-r^4} \right) - \sin\left(\dfrac{\pi \epsilon}{2}\right) =0$;
			
			\item  $2 \sum_{n\geq1}\left( \dfrac{\psi^2_{\beta,q,n}(a) r^2}{\psi^4_{\beta,q,n}(a)-r^4}+ \sin\left( \dfrac{\pi \epsilon}{2}\right) \dfrac{r^4}{\psi^4_{\beta,q,n}(a)-r^4} \right) - \sin\left(\dfrac{\pi \epsilon}{2}\right) =0$;

			\item $\sum_{n\geq1}\left( \dfrac{{\psi}^2_{\beta,q,n}(a) r}{{\psi}^4_{\beta,q,n}(a)-r^2}+ \sin\left( \dfrac{\pi \epsilon}{2}\right) \dfrac{r^2}{{\psi}^4_{\beta,q,n}(a)-r^2} \right) - \sin\left(\dfrac{\pi \epsilon}{2}\right) =0$
			
		\end{enumerate}
		in $(0,\psi_{\beta,q,1})$, $(0,\psi_{\beta,q,1})$ and $(0,\psi^2_{\beta,q,1})$ respectively.
	\end{theorem}

\section*{Conflict of interest}
	The authors declare that they have no conflict of interest.


\begin{thebibliography}{99}
	\setlength{\itemsep}{5pt}
\bibitem{abo-2018}Akta\c{s}, Baricz\ and\ H. Orhan, Bounds for radii of starlikeness and convexity of some special functions, Turkish J. Math. {\bf 42} (2018), no.~1, 211--226.

\bibitem{btk-2018}Baricz, E. Toklu\ and\ E. Kadio\u{g}lu, Radii of starlikeness and convexity of Wright functions, Math. Commun. {\bf 23} (2018), no.~1, 97--117. 

\bibitem{b-praj-2020}Baricz\ and\ A. Prajapati, Radii of starlikeness and convexity of generalized Mittag-Leffler functions, Math. Commun. {\bf 25} (2020), no.~1, 117--135.

\bibitem{bdoy-2016}Baricz\ et al., Radii of starlikeness of some special functions, Proc. Amer. Math. Soc. {\bf 144} (2016), no.~8, 3355--3367.

\bibitem{Bricz-Rama}Baricz\ and\ N. Ya\u{g}mur, Geometric properties of some Lommel and Struve functions, Ramanujan J. {\bf 42} (2017), no.~2, 325--346.

\bibitem{bulut-engel-2019} S. Bulut\ and\ O. Engel, The radius of starlikeness, convexity and uniform convexity of the Legendre polynomials of odd degree, Results Math. {\bf 74} (2019), no.~1, Paper No. 48, 9 pp.

\bibitem{sinefun}N. E. Cho\ et al., Radius problems for starlike functions associated with the sine function, Bull. Iranian Math. Soc. {\bf 45} (2019), no.~1, 213--232.

\bibitem{ErhanDenij2020} E. Deniz,  Geometric and monotonic properties of Ramanujan type entire functions. Ramanujan J, 2020.

\bibitem{Deniz-2017}E. Deniz\ and\ R. Sz\'{a}sz, The radius of uniform convexity of Bessel functions, J. Math. Anal. Appl. {\bf 453} (2017), no.~1, 572--588.

\bibitem{lp}D. K. Dimitrov\ and\ Y. Ben Cheikh, Laguerre polynomials as Jensen polynomials of Laguerre-P\'{o}lya entire functions, J. Comput. Appl. Math. {\bf 233} (2009), no.~3, 703--707.

\bibitem{ganga1997}A. Gangadharan, V. Ravichandran\ and\ T. N. Shanmugam, Radii of convexity and strong starlikeness for some classes of analytic functions, J. Math. Anal. Appl. {\bf 211} (1997), no.~1, 301--313.

\bibitem{goel2020}P. Goel\ and\ S. Sivaprasad Kumar, Certain class of starlike functions associated with modified sigmoid function, Bull. Malays. Math. Sci. Soc. {\bf 43} (2020), no.~1, 957--991.

\bibitem{ismail-2018}M. E. H. Ismail\ and\ R. Zhang, $q$-Bessel functions and Rogers-Ramanujan type identities, Proc. Amer. Math. Soc. {\bf 146} (2018), no.~9, 3633--3646.

\bibitem{janow}W. Janowski, Extremal problems for a family of functions with positive real part and for some related families, Ann. Polon. Math. {\bf 23} (1970/71), 159--177.

\bibitem{kanas}S. Kanas\ and\ A. Wi\'{s}niowska, Conic domains and starlike functions, Rev. Roumaine Math. Pures Appl. {\bf 45} (2000), no.~4, 647--657 (2001).

\bibitem{SG-2020}S. S. Kumar and K. Gangania, Subordination and radius problems for certain starlike functions, arXiv:2007.07816

\bibitem{bohr}S. S. Kumar and K. Gangania, On Certain Generalizations of $\mathcal{S}^*(\psi)$, arXiv:2007.06069

\bibitem{pathan-2016} H. Kumar and A. M. Pathan, On the distribution of non-zero zeros of generalized Mittag-Leffler functions, Int. J. Eng. Res. Appl. {\bf 6} (2016), no.~10, 66--71. 

\bibitem{Levin-1996}B. Ya. Levin, {\it Lectures on entire functions}, translated from the Russian manuscript by Tkachenko, Translations of Mathematical Monographs, 150, American Mathematical Society, Providence, RI, 1996.

\bibitem{minda94}W. C. Ma\ and\ D. Minda, A unified treatment of some special classes of univalent functions, in {\it Proceedings of the Conference on Complex Analysis (Tianjin, 1992)}, 157--169, Conf. Proc. Lecture Notes Anal., I, Int. Press, Cambridge, MA.

\bibitem{vibha} V. Madaan, A. Kumar and V. Ravichandran.: Radii of Starlikeness and Convexity of Bessel Functions, 	arXiv:1906.05547 

\bibitem{mendi} R. Mendiratta, S. Nagpal\ and\ V. Ravichandran, A subclass of starlike functions associated with left-half of the lemniscate of Bernoulli, Internat. J. Math. {\bf 25} (2014), no.~9, 1450090, 17 pp.

\bibitem{mendi2exp}R. Mendiratta, S. Nagpal\ and\ V. Ravichandran, On a subclass of strongly starlike functions associated with exponential function, Bull. Malays. Math. Sci. Soc. {\bf 38} (2015), no.~1, 365--386.

\bibitem{pap-sokol1996}E. Paprocki\ and\ J. Sok\'{o}\l, The extremal problems in some subclass of strongly starlike functions, Zeszyty Nauk. Politech. Rzeszowskiej Mat. No. 20 (1996), 89--94.

\bibitem{prabha-1971}T. R. Prabhakar, A singular integral equation with a generalized Mittag Leffler function in the kernel, Yokohama Math. J. {\bf 19} (1971), 7--15.

\bibitem{singh-1968}R. Singh, On a class of star-like functions, Compositio Math. {\bf 19} (1967), 78--82 (1967).

\bibitem{watson-1944}G. N. Watson, {\it A Treatise on the Theory of Bessel Functions}, Cambridge University Press, Cambridge, England, 1944.

\end{thebibliography}
\end{document}